\documentclass[a4paper,11pt]{amsart}

\input{macros}

\title[A Univalent Formalization of Constructive Affine Schemes]{A Univalent Formalization of Constructive Affine Schemes}
\author{Max Zeuner and Anders M\"ortberg}
\address{Department of Mathematics, Stockholm University, Sweden}
\email{\href{mailto:zeuner@math.su.se}{zeuner@math.su.se} and \href{mailto:anders.mortberg@math.su.se}{anders.mortberg@math.su.se}}

\date{\today}

\thanks{Funding: this paper is based upon research supported by the
  Swedish Research Council (SRC, Vetenskapsrådet) under Grant
  No.~2019-04545. The research has also received funding from the Knut
  and Alice Wallenberg Foundation through the Foundation's program for
  mathematics.}

\begin{document}

\maketitle

\begin{abstract}
We present a formalization of constructive affine schemes in the \CubicalAgda~proof assistant.
This development is not only fully constructive and predicative,
it also makes crucial use of univalence.
By now schemes have been formalized in various proof assistants.
However, most existing formalizations follow the inherently
non-constructive approach of Hartshorne's classic ``Algebraic Geometry'' textbook,
for which the construction of the so-called structure sheaf is rather straightforwardly
formalizable and works the same with or without univalence.
We follow an alternative approach
that uses a point-free description of the constructive counterpart of
the Zariski spectrum called the Zariski lattice
and proceeds by defining the structure sheaf on formal basic opens
and then lift it to the whole lattice.
This general strategy is used in a plethora of textbooks, but formalizing it has
proved tricky. The main result of this paper is that with the help of
the univalence principle we can make this
``lift from basis'' strategy formal and obtain a fully formalized
account of constructive affine schemes.
  %% We will focus on the construction of the structure sheaf on basic opens
  %% as it is at this point that working in a univalent setting raises
  %% perhaps surprising issues while at the same time offering insightful
  %% solutions.
\end{abstract}

\section{Introduction}

Algebraic geometry originated as the study of solutions of
polynomials. Historically, the geometric objects of interest would be
for example \emph{complex affine varieties}---subsets of
$\mathbb{C}^n$ defined by systems of polynomial equations.  Starting
with the pioneering work of Grothendieck in the 1960s, the scope of
the discipline was drastically widened, making it one of the most
pervasive in modern day mathematics. At the heart of this development
are \emph{schemes}---geometric objects that generalize from
algebraically closed fields, like $\mathbb{C}$, to arbitrary
commutative rings.

A point $a\in\mathbb{C}$ corresponds to the maximal ideal of the
polynomial ring $\mathbb{C}[x]$ consisting of polynomials $p$ such
that $p(a)=0$, \ie\ the ideal generated by
$(x-a)$. By looking not only at maximal ideals, but also at prime
ideals of $\mathbb{C}[x]$, we arrive at the \emph{spectrum} of
$\mathbb{C}[x]$, denoted $\Spec{\mathbb{C}[x]}$. As $\mathbb{C}[x]$
has a non-maximal prime ideal, the zero-ideal, $\Spec{\mathbb{C}[x]}$
contains an additional point to $\mathbb{C}$ and carries a very different
topology. This is called the \emph{Zariski topology} in which the open
sets are generated by \emph{basic opens} $D(p)\subseteq\Spec{\mathbb{C}[x]}$
where $p\in\mathbb{C}[x]$.
If $p\neq 0$, $D(p)$ corresponds to the set
of points $a$ where $p(a)\neq 0$ together with the zero-ideal.
The spectrum can then be equipped with a
\emph{structure sheaf} that associates to every Zariski open set $U$ a ring of
``rational functions'' definable on $U$. For a basic open $D(p)$, this
will be the ring of function $\nicefrac{q(x)}{p(x)^n}$ where $q$ is
another polynomial.  This corresponds to the functions of the quotient
ring $\mathbb{C}(x)$ that are definable everywhere but at the zeros of
$p$. See Vakil's ``The Rising Sea'' \cite[Ex. 3.2.3.1]{RisingSea} for a more in-depth
discussion of this motivating example and an illustration of $\Spec{\mathbb{C}[x]}$.

This construction can be carried out for any commutative ring $R$
instead of $\mathbb{C}[x]$: the spectrum $\Spec{R}$ is the set of
prime ideals of $R$ and its Zariski topology is again generated by
basic opens. For $f\in R$, the basic open $D(f)$ is the set of prime
ideals that do \emph{not} contain $f$. The structure sheaf maps $D(f)$
to the \emph{localization} $\locEl{R}{f}$, the ring of fractions
$\nicefrac{r}{f^n}$ where $r\in R$ and the denominator is a power of
$f$. One can prove that this always defines a sheaf, i.e.\ is
compatible with taking covers of open sets in a certain sense. % , for any
% $R$.

When Grothendieck introduced the general notion of (affine) schemes,
he did so in a structural fashion that is typical for his work.
Mathematical objects, in particular algebraic structures, are taken to
be identical if they are isomorphic in some unique, or at least
canonical, way. When constructing the structure sheaf, however, this
leads to a problem of well-definedness: if $D(f) = D(g)$, then we
better have $\locEl{R}{f}=\locEl{R}{g}$.  Unfortunately, it is not
difficult to come up with examples violating this.  For example, we
have $D(x) = D(x^2)$ in $\mathbb{C}[x]$ (both functions vanish only at
$0$), but formally speaking $\locEl{\mathbb{C}[x]}{x}$ is not strictly
the same ring as $\locEl{\mathbb{C}[x]}{x^2}$ despite them clearly
being isomorphic and describing the same sub-ring of the quotient ring $\mathbb{C}(x)$,
as $\nicefrac{1}{x} = \nicefrac{x}{x^2}$.

In this paper we show how this problem can be solved with the help of
univalence.  In particular, we present a formalization in
\CubicalAgda~\cite{CubicalAgda2} of constructive affine schemes
following Coquand, Lombardi and Schuster \cite{ConstrSchemes}. In the
constructive setting, the Zariski spectrum of a commutative ring is
replaced by the so-called \emph{Zariski lattice}. Elements of this
lattice are \emph{finitely} generated by formal basic opens, which
allows for a completely predicative approach that does not require
additional assumptions like Voevodsky's resizing axioms
\cite{VoevodskyResizing}.

The definition of constructive affine schemes still works analogously
to the classical definition given in most textbooks ranging from
Grothendieck's authoritative classic ``EGA I'' \cite{EGA1}, to more
modern treatments such as ``Algebraic Geometry'' by G\"ortz and
Wedhorn \cite{GoertzWedhorn}, ``The Rising Sea'' by Vakil
\cite{RisingSea}, or Johnstone's ``Stone Spaces''
\cite{StoneSpaces}. In either case one starts with the basic opens, on
which the structure sheaf is defined and proved to be a sheaf. Using
abstract categorical machinery this is then lifted to a sheaf on the
whole Zariski spectrum/lattice.  More precisely, one takes the right
Kan extension along the inclusion of basic opens, which preserves the
sheaf property.

From a constructive, predicative point of view there are two
differences that make this construction work for the Zariski lattice.
Predicatively, the inclusion of basic opens into the Zariski lattice
is one of small categories, while the inclusion into the classical
spectrum is not.  Furthermore, since we are only concerned with
sheaves on a distributive lattice and not on a general locale or
topological space, we only have to consider finite covers.  This
allows for a predicative proof that the right Kan extension preserves
sheaves on lattices. From a classical point of view this is not really
a restriction as $\Spec{R}$ is always a \emph{coherent} space.
As a result, sheaves on $\Spec{R}$ are in bijection to (finitary) sheaves
on the Zariski lattice. For more details see
\eg\ Johnstone's ``Stone Spaces'' \cite{StoneSpaces} and
\cref{subsec:classicalcomparison} of this paper.

%% \anders{It is strange to start a paragraph with ``However''. Maybe
%%   something like ``Regardless whether one formalizes the classical or
%%   constructive definitions one runs into challenges related to
%%   well-definedness of the structure sheaf.''}
%

Regardless of whether one formalizes the classical or constructive definitions,
the main bottlenecks of the formalization
are already found at the level of basic opens.
First and foremost, there is the well-definedness problem described above.
The second bottleneck is
proving that the structure sheaf actually is a sheaf on basic opens. In
fact, the problem with the textbook proof of the sheaf property is the
well-definedness problem in disguise.  Those two points were exactly
where the most prominent formalization of schemes \cite{SchemesLean}
in \Lean's \mathlib \cite{Lean/mathlib} encountered problems.  In this
paper, we show that with the help of univalence it is in fact possible
to overcome the issues of well-definedness and formalize the
structure sheaf directly on basic opens and prove its sheaf property.
Even though we work in the constructive, predicative setting using the
Zariski lattice, the techniques used to overcome the problems on the
level of basic opens should be applicable to a classical formalization
in type theory with univalence and classical axioms added. The key
insight is that localizations are not just commutative rings, but also
commutative algebras over $R$. In $R$-algebras, isomorphisms, and thus
also paths, between two localizations are unique, which ensures
well-definedness of the structure sheaf.

%% It turns out that the problems to overcome
%% are subtle and that the construction is indeed more involved
%% than one might expect.
%% In particular
%% It turns out that, indeed, it is not sufficient to provide
%% a path between the rings $\locEl{R}{f}$ and $\locEl{R}{g}$ whenever
%% $D(f)=D(g)$. The key insight is that localizations are not just commutative
%% rings, but also commutative algebras over $R$. So, on basic opens the
%% structure sheaf takes values in $R$-algebras and we can prove,
%% using univalence, that
%% % \begin{center}
%%   $D(f)= D(g)$ iff the path space (in $R$-agebras) between $\locEl{R}{f}$
%%   and $\locEl{R}{g}$ is contractible.
%% % \end{center}
%% This ensures the well-definedness of the structure sheaf, which we
%% obtain by composing with the forgetful functor from $R$-algebras to
%% commutative rings. It also lets us
%% follow the textbook proof that the structure sheaf is indeed a
%% sheaf, using a standard lemma in commutative algebra and by transporting along
%% (unique) paths of $R$-algebras.

As mentioned above, our work is completely
formalized\footnote{\label{footnote:formalization}All results discussed are
  integrated in the \CubicalAgda library and are summarized in:\\
  \url{https://github.com/agda/cubical/blob/310a0956bb45ea49a5f0aede0e10245292ae41e0/Cubical/Papers/AffineSchemes.agda} \\
  This is a permalink to the library at the time of writing, which type-checks with
  \Agda~version 2.6.3. \\
  A clickable rendered version that might be subject to change can be found  here:\\
  \url{https://agda.github.io/cubical/Cubical.Papers.AffineSchemes.html}}
in \CubicalAgda, an extension of the \Agda proof assistant \cite{Agda}
based on the cubical type theory of
\cite{CCHM18,CoquandHuberMortberg18} with fully constructive support
of the univalence axiom and higher inductive types (HITs). However,
nothing relies crucially on cubical features, or on univalence and
eliminators applied to higher constructors of HITs computing
definitionally, in our formalization. The only {HoTT/UF} features
that we rely on are univalence and set quotients (from which
propositional truncation follows). It would hence be possible to
perform the formalization in a system implementing Book
HoTT~\cite{HoTTBook} or in \UniMath \cite{UniMath}.
Our work is thus in line with the aim of
Voevodsky's Foundations library \cite{VoevodskyFoundationsLib} of
developing a library of constructive set-level mathematics based on
Univalent Foundations.

\paragraph*{Contributions}
As mentioned above, the formalization presented in this paper
generally follows the constructive, lattice-based approach of
\cite{ConstrSchemes}. However, a number of design choices had to be
made to ensure predicativity of our formalization and
to enable us to formally prove the well-definedness of the structure sheaf.
As a result some definitions and proofs deviate from the presentation
in \cite{ConstrSchemes}. The main design choices and contributions of
the paper and formalization can be summarized under the following topics:

\begin{itemize}

\item \textbf{Commutative algebra:} our formalization of localizations
  of commutative rings in \cref{subsec:localizations} closely follows
  Atiyah and MacDonald's classic textbook \cite{AtiyahMacDonald},
  which works very well for our constructive approach. However, giving
  a predicative definition of the Zariski lattice that does not
  increase universe level was more intricate. To this end
  \cref{ZarLat} contains a construction that refines the ideal-based
  description of \cite{ConstrSchemes} using ideas of Espa\~nol
  \cite{Espanol83}.

\item \textbf{Category theory:} in \cref{CatTheory} we present a
  formal notion of sheaf on a distributive lattice that closely
  follows \cite{ConstrSchemes}. However, in \cite{ConstrSchemes}
  presheaves are extended from a basis of a distributive lattice to
  the whole lattice in a somewhat non-standard finitary way. This is
  to ensure predicativity, but it actually causes problems when
  working in a univalent setting. We found that the point-wise right
  Kan extension of presheaves, as \eg\ presented in MacLane's classic
  textbook \cite{MacLaneCategories}, works just fine even in the
  constructive and predicative setting. We then give a proof that the
  Kan extension preserves the sheaf property. This can be seen as the
  main step towards a constructive and predicative ``comparison
  lemma'' that gives an equivalence of categories between sheaves on a
  lattice and sheaves on the basis of the lattice.

\item \textbf{Constructive affine schemes:} in
  \cref{sec:structuresheaf} we construct
  the structure sheaf on basic opens and extend it to the Zariski lattice.
  We give general heuristics for constructing
  presheaves (valued in $R$-algebras) on subsets defined using
  propositional truncation. The well-definedness of the presheaves thus
  constructed follows from univalence.
  The structure sheaf is a special instance
  of this construction with the basic opens seen as a subset of the
  Zariski lattice. Proving the sheaf property on basic opens can then
  be reduced to standard commutative algebra, again by using
  univalence in a way that does not require to extract the isomorphisms underlying
  the applications of univalence.

\end{itemize}

% In the conclusion we give further context to our work by discussing
% the relation to existing formalization and the relation of the
% constructive approach to the classical standard approach.

\section{Background}
\label{sec:background}

Here we give the necessary background for the paper.
We first sketch the constructive approach to schemes of
\cite{ConstrSchemes}.
%% In particular, we highlight some of the
%% differences to classical approaches, while stressing important
%% similarities.
We then continue with an introduction to the concepts of
\CubicalAgda needed for the paper.

\subsection{Affine schemes constructively}

Recall that, classically, the spectrum of a commutative ring $R$ is the set of its prime ideals
$\Spec R=\{\mathfrak{p}\subseteq R\;\vert\;\mathfrak{p}\text{ prime}\}$
equipped with the Zariski topology. The open sets of this topology
are generated by basic opens $D(f)=\{\mathfrak{p}\;\vert\; f\notin\mathfrak{p}\}$
for $f\in R$. Constructively, there are two issues with this.
First, the notion of prime ideal is not really well-behaved.
One of the main reasons for this is that the central notion of \emph{localizing} at a prime ideal
$\mathfrak{p}$ actually uses the set-theoretic complement $R\setminus \mathfrak{p}$,
which does not work well constructively without additional decidability assumptions.\footnote{
  See e.g.\ the discussion by Mines, Richman and Ruitenberg in their standard textbook on
  constructive algebra \cite[Section III.3]{MinesRichman}.}
To remedy this, one can define the notion of a \emph{prime filter} on $R$
and check that classically those are exactly the complements of prime ideals.

The second issue concerns the point-set definition of a topological space itself.
For a constructive development of algebraic geometry it is preferable to avoid this definition
and instead characterize the \emph{locale} of open sets
of $\Spec R$ in a direct, point-free way. This can be done by observing that
the closed sets of the Zariski topology admit a direct algebraic characterization.
Every closed set is of the form
$V(\mathfrak{a})=\{\mathfrak{p}\;\vert\;\mathfrak{a}\subseteq\mathfrak{p}\}$,
where $\mathfrak{a}$ is a \emph{radical ideal}. An ideal $\mathfrak{a}\subseteq R$
is radical if $\mathfrak{a}=\sqrt{\mathfrak{a}}$, where
\begin{align*}
  \sqrt{\mathfrak{a}}=\big\{\, x\in R \;\vert\; \exists n>0:x^n\in\mathfrak{a}\,\big\}
\end{align*}
The \emph{locale of Zariski opens} can thus be characterized by the set of radical ideals
of $R$. The join and meet operation can be defined using addition and multiplication of ideals.

From a predicative viewpoint this is still unsatisfactory. Predicatively, the ideals of a ring
form a proper class and consequently the Zariski locale is not a set in such a setting.
However, by restricting to the \emph{lattice of compact open sets} of the Zariski topology
these size issues can be avoided.\footnote{Through a more careful analysis
  one \emph{might} be able to define the structure sheaf on the large Zariski locale
  in predicative univalent foundations, as long as one uses a small type of basic opens.
  See the recent work by de Jong and H\"otzel Escard\'o \cite{DeJongEscardoSmall}
  and by Tosun and H\"otzel Escard\'o \cite{TosunEscardo} for results of this kind.
  For the development of constructive and predicative scheme theory however, it seems
  certainly advantageous to work with the small Zariski lattice.}
Classically, the objects of this lattice are finite
unions of basic opens $D(f_1)\cup\dots\cup D(f_n)$ and the join and meet operation
are just union $\cup$ and intersection $\cap$. Note that for the meet this only works
because basic opens are closed under intersections,
i.e.\ we have $D(f)\cap D(g)=D(fg)$ for any $f,g\in R$.

As with the locale of Zariski opens, this so-called \emph{Zariski lattice}
\ZL~of a commutative ring $R$ can be described in a point-free way.
This was first done by Joyal \cite{JoyalZarLat},
using the observation that the Zariski lattice has a certain universal property.
The lattice itself can be defined as the free distributive lattice generated by
\emph{formal symbols} $D(f)$, $f\in R$, satisfying the following relations:
\begin{align}
  & D(1)=\top \;\text{ and }\; D(0)=\bot \\
  & \forall f,g\in R:\;D(fg)=D(f)\wedge D(g) \\
  & \forall f,g\in R:\;D(f+g)\leq D(f)\vee D(g)
\end{align}
The induced map $D:R\to\ZL$ is universal in the following sense:
for any distributive lattice $L$ and \emph{support} map $d:R\to L$,
i.e.\ any map $d$ such that conditions (1)-(3) above hold for  $d$ (in place of $D$),
there is a unique lattice homomorphism $\varphi:\ZL\to L$ such that
the following commutes
\[
\begin{tikzcd}
  & R \arrow[dl,"D"']\arrow[dr,"d"] & \\
  \ZL \arrow[rr,dashed, "\exists!~\varphi"'] && L
\end{tikzcd}
\]
Using the correspondence of Zariski opens with radical ideals, the elements of \ZL~can
also be described as the \emph{radicals of finitely generated ideals}.
For two finitely generated ideals $\mathfrak{a},\mathfrak{b}\subseteq R$,
the join and meet of the radicals are then given by
\begin{align*}
  \sqrt{\mathfrak{a}}\vee \sqrt{\mathfrak{b}}=\sqrt{\mathfrak{a}+\mathfrak{b}}
  \quad\quad\text{and}\quad\quad
  \sqrt{\mathfrak{a}}\wedge \sqrt{\mathfrak{b}}=\sqrt{\mathfrak{a}\mathfrak{b}}
\end{align*}
using the fact that addition and multiplication of two finitely generated ideals
is again finitely generated. The support $D:R\to\ZL$ maps $f\in R$ to the radical
of the principal ideal $\sqrt{\langle f\rangle}$ and for any support $d:R\to L$,
the unique morphism $\varphi:\ZL\to L$ is given by
\begin{align*}
  \varphi\big(\sqrt{\langle f_1,\dots,f_n\rangle}\;\big)\;=\; d(f_1) \vee\dots\vee d(f_n)
\end{align*}
In \cref{ZarLat}, we will show how to formalize this Zariski lattice of radicals
of finitely generated ideals and prove its universal property while avoiding size issues.

The lattice theoretic approach does require a notion of a sheaf on a distributive lattice.
Recall that a sheaf on a topological space $X$ is just a sheaf on the locale
of open sets of $X$. By restricting the definition of sheaf on a locale to finite covers one
obtains sheaves on a distributive lattice. This means that for any distributive lattice $L$,
a presheaf $\mathcal{F}:L^{op}\to \mathcal{C}$, valued \eg\ in commutative rings (\ie{} $\mathcal{C} = \mathsf{CommRing}$),
is a sheaf if for all $x_1,\dots,x_n\in L$ the following is an equalizer diagram
\begin{align*}
  \mathcal{F}\Big(\bigvee_{i=1}^n x_i\Big) \to \prod_{i=1}^n\mathcal{F}(x_i)\rightrightarrows
                                              \prod_{i<j}\mathcal{F}(x_i \wedge x_j)
\end{align*}
A basis of a distributive lattice is a subset $B\subseteq L$ containing $\top$
and closed under meets, such that for any $x\in L$
there exists a finite list $b_1,\dots,b_n\in B$ such that $x=\bigvee_{i=1}^nb_i$.
In \cref{CatTheory}, we describe how to obtain sheaves on $L$ from sheaves on $B$.
This works analogous to the special case of the
so-called \emph{comparison lemma} for topological spaces.
For the structure sheaf, the idea is to map $D(f)$  to the ring $\locEl{R}{f}$,
the \emph{localization of $R$ away from $f$}.
Recall that for a subset $S\subseteq R$ containing $1$ and being closed under multiplication,
the localization $S^{-1}R$ is defined as the ring of fractions $\nicefrac{r}{s}$ where $r\in R$
and $s \in S$. Equality of fractions is given by
\begin{align*}
  \frac{r_1}{s_1}=\frac{r_2}{s_2} \quad \text{iff} \quad \exists u\in S:\; u(r_1s_2-r_2s_1)=0
\end{align*}
$\locEl{R}{f}$ is defined by localizing with $S=\{1,f,f^2,f^3,\dots\}$. Its elements
are thus fractions $\nicefrac{r}{f^n}$ where the denominator is a power of $f$ and
equality can rephrased as
\begin{align*}
  \frac{r}{f^n}=\frac{r'}{f^m} \quad \text{iff} \quad \exists k\in\mathbb{N}:\; f^{k+m}r=f^{k+n}r'
\end{align*}
Verifying that the presheaf defined by sending $D(f)$ to $\locEl{R}{f}$ is indeed a sheaf
on the basis $\BO\subseteq\ZL$ of basic opens
proceeds the same way in any constructive or classical account.
As indicated in the introduction, there are some issues to be overcome
when formalizing the construction of the structure sheaf. In this
paper we discuss how a solution to these problems can look like in a
univalent setting.\footnote{A solution that is \eg\ taken in \cite{ConstrSchemes},
  is to map $D(f)$ to $S_f^{-1}R$, the ring of fractions whose denominators
  are elements of $S_f=\{g\,\vert\, D(f)\subseteq D(g)\,\}$.
  It is immediate to see that if $D(f)=D(g)$, then $S_f^{-1}R=S_g^{-1}R$,
  but it is not as natural to work with these rings.
  Usually one still wants to appeal to the ``canonical isomorphism''
  between $\locEl{R}{f}$ and $S_f^{-1}R$, as in \eg\ \cite[Sect. 1.3]{EGA1}.}

\subsection{Set-level univalent mathematics in CubicalAgda}
We will now briefly discuss the concepts needed from \CubicalAgda for
this paper, for more details see \cite{CubicalAgda2}.
Our notation is inspired by \Agda syntax and
the \systemname{agda/cubical} library,
but we have taken some liberties when typesetting, e.g.\ shortening
notations and omitting some projections and universe levels whenever possible.
We write $\func{Type}~\var{ℓ}$ for universes (at level~$\ell$) and
$\tySigmaNoParen{x}{A}{B(x)}$
for dependent pair types over a family $B:A\to\Type~\ell$.
The major difference when working in \CubicalAgda compared to vanilla
\Agda or Book HoTT is that the primary identity type % over a type $A$
is changed from Martin-Löf's inductive construction
\cite{MartinLof75itt} to a primitive \emph{path}-type. The
identification $x \,\func{$\equiv$}\,y$ is captured by
$\func{Path}\,A\,x\,y$, the type of functions $p : \func{I} \to A$,
where $\func{I}$ is a primitive interval type, restricting
definitionally to $x$ and $y$ at the endpoints \cons{i0} and \cons{i1}
of \func{I}.
\CubicalAgda also has a dependent path type, \func{PathP}. Given a
line of types \var{B} : \func{I} $\to$ \func{Type}, which we
 may think of as \var{B}(\cons{i0}) \func{$\equiv$} \var{B}
(\cons{i1}), and \var{x} : \var{B}(\cons{i0}), \var{y} : \var{B}
(\cons{i1}), the type \func{PathP} \var{B} \var{x} \var{y} expresses
that \var{x} and \var{y} may be identified relative to \var{B}. The
regular %% (homogeneous)
path type $\func{\_$\equiv$\_}$ is, by
definition, \func{PathP} ($\lambda$ \var{i} $\to$ \var{A}), \ie\ the
special case of a constant line of types.

\CubicalAgda also comes with a function
$\func{ua} : A\,\func{≃}\,B \to A\,\func{≡}\,B$ which promotes
equivalences (or isomorphisms) of types to paths between these
types. The fact that this map is an equivalence itself is a way to
formulate Voevodsky's univalence axiom.
A reasonable question to ask in a univalent setting is
whether an equivalence of types can be promoted to an equality of
structured types, such as groups or rings. The \emph{Structure
  Identity Principle (SIP)} \cite[Sect.~9.8]{HoTTBook} is an
informal principle which attempts to answer this: given two structured
types $(A,S_A)$ and $(B,S_B)$ and an equivalence of underlying types
$A \simeq B$ which is a homomorphisms with respect to the structure in
question, we get a path of structured types
$(A,S_A)\,\func{$\equiv$}\,(B,S_B)$. For instance, an isomorphism of
rings $R$ and $S$ induces a path $R\,\func{$\equiv$}\,S$. This has
been implemented in
\systemname{agda/cubical} using the cubical SIP of
Angiuli,~Cavallo,~Mörtberg~and~Zeuner~\cite{POPLPaper}. For this paper
we will use \func{sip} to denote the function
that turns isomorphisms of commutative rings or $R$-algebras (over a ring $R$) into paths.

Univalence refutes \emph{Uniqueness of Identity Proofs (UIP)}, or
Streicher's axiom K \cite{Streicher93}, because it produces equality
proofs in \func{Type} that are not equal \cite[Ex. 3.1.9]{HoTTBook}. In the presence
of univalence, it is therefore important to keep track of which types
satisfy UIP or related principles expressing the complexity of a
type's equality relation. In the terminology of HoTT/UF, a type
satisfying UIP is called an \emph{h-set} (\emph{homotopy set}, henceforth
simply \emph{set}), while a type whose elements are all equal is
called an \emph{h-proposition} (henceforth \emph{proposition}).

Another very important concept in HoTT/UF is that of
\emph{contractible} types, \ie\ types with exactly one element:
\ExecuteMetaData[agda/latex/Background.tex]{contr}
We can characterize propositions as types whose equality types are
contractible, just as sets are types whose equality types are
propositions. Thus contractible types, propositions, and sets serve as
the bottom three layers of an infinite hierarchy of types introduced
by Voevodsky, known as \emph{h-levels} \cite{Voevodsky10bonn} or
\emph{$n$-types} \cite{HoTTBook}. This paper is about set-level
mathematics, so we are mainly interested in these 3 bottom layers.
However, univalence implies that collections of set-level
structures (\eg\ the collection of all commutative rings or
$R$-algebras) are one level higher than sets. Types at this level
are called \emph{h-groupoids} (heceforth \emph{groupoids}) and will be the
only types of h-level higher than $2$ in the paper.
We write $\func{isProp}~A$ to say that $A$ is a
proposition and $\func{isSet}~A$ to say that $A$ is a set.
The ``universe of propositions'' $\func{hProp}~\ell$ is defined as
$\tySigma{A}{\Type~\ell}{\func{isProp}~A}$,
and if $\func{isSet}~A$, we call functions $S:A\to\hProp~\ell$
a \emph{subset} of $A$. For $a:A$ we denote by $a~\func{$\in$}~S$
the type of proofs that $a$ is actually in $S$. It is often convenient
to identify the subset $S$ with $\tySigma{a}{A}{a~\func{$\in$}~S}$,
which can be seen as a sub-type of $A$. With some abuse of notation
we will not distinguish between subsets as functions and the
corresponding $\Sigma$-type. We thus write $a:S$ for elements
of $S$ when the proof of $a$ belonging to $S$ can be ignored.

Another concept from HoTT/UF which \CubicalAgda supports are higher
inductive types (HITs). These allow us to define many important
operations on types, such as truncations. For instance, the
propositional truncation is defined by: % $(-1)$-truncation
\ExecuteMetaData[agda/latex/Background.tex]{propTrunc}
This HIT takes a type $A$ and forces it to be a proposition. This is a
very important construction for capturing existential quantification
in HoTT/UF:
\begin{align*}
  \tyExistsNoParen{x}{A}{P(a)} = \ptrunc{\tySigmaNoParen{x}{A}{P(a)}}
\end{align*}
In this paper, we follow the HoTT Book terminology and say that $a$
\emph{merely} exists when it is existentially quantified. Also, note
that the propositional truncation in the definition is crucial. In
HoTT/UF, $\func{Σ}\,A\,P$ without the truncation, is interpreted as the
total space of $P$, which may be highly non-trivial. For example,
a subset $B$ of a lattice $L$ is a basis if
\begin{align*}
  \forall\,(x:L)\to\tyExists{b_1,\dots,b_n}{B}{\tyPath{\textstyle\bigvee_{i=1}^nb_i}{x}}
\end{align*}
Here the propositional truncation is crucial as we will see
in \cref{ZarLat}, when proving that the basic opens
form a basis of \ZL.

The main HIT that we use in this paper is the \emph{set quotient}, which
quotients a type by an arbitrary relation, yielding a set. It has three
constructors: \cons{[\_]}, which includes elements of the underlying type,
\cons{eq/}, which equates all pairs of related elements, and \cons{squash/}, which
ensures that the resulting type is a set:
\ExecuteMetaData[agda/latex/Background.tex]{setquot}
We can write functions out of $A\;\func{/}\;R$ by pattern-matching;
this amounts to writing a function out of $A$ (the clause for $\cons{[\_]}$) which
sends $R$-related elements of $A$ to equal results (the clause for $\cons{eq/}$),
such that the image of the function is a set (the clause for $\cons{squash/}$).
Set quotients and propositional truncations have in common that the resulting
type will be of a fixed h-level and this makes it very hard to map into types
of higher h-levels. In fact, the higher the h-level of the target type,
the more complicated the coherence conditions that need to be proved.
We will see an example of this in \cref{sec:structuresheaf}.

\section{Commutative algebra}
\label{sec:algebra}

In this section, we first discuss our formalization of localizations of rings, followed by the definition of the Zariski lattice.
These objects can be described by universal properties,
but may also be concretely implemented as set quotients.
One of the guiding principles of this project was to work with concrete implementations
and mainly use universal properties to construct equivalences and paths
(via the SIP).
As a result the formalization follows the usual informal treatment
in the commutative algebra literature quite closely.

\subsection{Localizations}
\label{subsec:localizations}

Our formalization of localizations of commutative rings follows the classic
textbook of Atiyah and MacDonald \cite{AtiyahMacDonald},
with our main result being a path version of \cite[Cor. 3.2]{AtiyahMacDonald}.
Note that the definition of localization is actually the same
in classical and constructive algebra.\footnote{
  Compare \cite{AtiyahMacDonald} with \eg\ the books by Lombardi and Quitté \cite{LombardiQuitte}
  or Mines, Richman and Ruitenburg \cite{MinesRichman}.}
For the remainder of this paper we will only consider commutative rings with a
multiplicative unit (denoted by $1$). Let $R$ be such a ring and
$S$ a subset of $R$ that contains $1$
and is closed under multiplication.
The formalization of localization
is then straightforward:
\ExecuteMetaData[agda/latex/localisation.tex]{defLoc}
The underscores in the definition of \func{≈} correspond to the proofs that
$s_1,\,s_2$ and $s$ are elements of $S$ respectively. As these are unimportant
to the definition of \func{≈}, we can safely omit them.
\begin{remark}
  It might be surprising that we define \func{≈} using a
  $\func{Σ}$ and not an \func{∃} (as is done \eg\ in~\cite{VoevodskyFoundationsLib}).
  However, it turns out that it does not matter whether one quotients by the
  truncated relation using \func{∃} or the untruncated relation using \func{Σ},
  as the resulting set-quotients will be equivalent.
  As we do not need to prove anything about \func{≈}
  except it being an equivalence relation, it is more convenient to work without the truncation.
\end{remark}
Equipping $S^{-1}R$ with the structure of a commutative ring
is laborious in \CubicalAgda, but the proofs generally proceed as in any textbook.
The same holds for the universal property. Note that for this we need the
canonical homomorphism $\nicefrac{\_}{1}:R\to S^{-1}R$, mapping $r:R$ to
$\tmTrunc{\tmPair{r}{1}}$, the equivalence class corresponding to $\nicefrac{r}{1}$.
The universal property then states that for any commutative ring $A$
with a morphism $\varphi :R\to A$, such that for all $s:S$ we have $\varphi(s)\in A^\times$ (\ie\ that $\varphi(s)$ is a unit in $A$),
there is a unique morphism $\psi:S^{-1}R\to A$, such that the following commutes
\[
\begin{tikzcd}
  &R\arrow[ld,"\nicefrac{\_}{1}"']\arrow[rd,"\varphi"]&\\
  S^{-1}R\arrow[rr,dashed,"\exists!~\psi",swap]&&A
\end{tikzcd}
\]
The key observation for the main results of this paper is that localizations
are $R$-algebras via the canonical homomorphism $\nicefrac{\_}{1}$.
The type of $R$-algebras is equivalent to the $\Sigma$-type of a commutative ring $A$
together with a ring homomorphism $\varphi:R\to A$. An homomorphism between $R$-algebras
$(A,\varphi)$ and $(B,\psi)$
is just a ring homomorphism
$\chi:A\to B$ together with a path $\tyPath{\chi\circ \varphi}{\psi}$.
The type of $R$-algebra homomorphisms will be denoted by
$\text{Hom}_R\big[(A,\varphi)\,,(B,\psi)\big]$ or just
$\text{Hom}_R\big[A,B\big]$ if the morphisms are clear from context.

The universal property of localization then becomes a statement about $R$-algebras.
In HoTT/UF unique existence is defined as contractibility of
$\Sigma$-types,
so the universal property of the localization at $S$ becomes:
  \emph{for any $R$-algebra $(A,\varphi)$ s.t.\ $\varphi(S)\subseteq A^\times$,
    the type $\text{Hom}_R\big[S^{-1}R\,,(A,\varphi)\big]$ is contractible.}
%\end{center}
Combining the proof of \cite[Cor. 3.2]{AtiyahMacDonald} with the SIP,
we can then prove the following:\footnote{In \Lean's \mathlib \emph{a} localization is defined
  to be any ring-morphism-pair satisfying the
  three conditions of \cref{LocChar}.
  The formulation of this predicate
  is attributed
  to Neil
  Strickland in \cite{SchemesLean}.}
\begin{lemma}\label{LocChar}
  Let  $A$ be a commutative ring with a morphism $\varphi :R\to A$ satisfying
  \begin{itemize}
  \item $\forall (s:S)\to\varphi(s)\in A^\times$
  \item $\forall (r:R)\to\tyPath{\varphi(r)}{0}\to\tyExists{s}{S}{\tyPath{sr}{0}}$
  \item $\forall (a:A)\to\tyExists{(\tmPair{r}{s})}{R\times S}{\tyPath{\varphi(r)\varphi(s)^{-1}}{a}}$
  \end{itemize}
  From this we can construct a path $\tyPath{S^{-1}R}{A}$, which is unique as a path in $R$-algebras.
\end{lemma}
With this result we can transport proofs about localizations to any suitable
ring and morphism pair, \ie\ $R$-algebra, satisfying the three conditions above.
Below we will see a few applications of this result that will be used for
formalizing constructive affine schemes.
The important case for our purpose is $\locEl{R}{f}$,
the localization of $R$ \emph{away from} $f$.
This can be seen as inverting a single element $f$ in $R$.
The subset $S=\{1,f,f^2,f^3,...\}$ is easily defined in \CubicalAgda as
the set of $g:R$ for which we have an inhabitant of $\tyExists{n}{\func{ℕ}}{\tyPath{g}{f^n}}$.

For the remainder of this section let $f,g:R$.
By the canonical homomorphism we get an element $\nicefrac{g}{1}$ in $\locEl{R}{f}$.
With a bit of abuse of notation we denote the localization away from this element by
$\locEl{\locEl{R}{f}}{g}$.
This is an $R$-algebra by applying the canonical morphism $\nicefrac{\_}{1}$ twice.
We can of course also localize away from $(f~\field{·}~g)$, thus obtaining $\locEl{R}{fg}$.
Using \cref{LocChar}, we can construct a (unique) path between these two,
which will be used for the structure sheaf.
Similarly, we also get other useful paths.
\begin{lemma}\label{InvElIds}
  We have the following paths for both commutative rings and $R$-algebras:
  \begin{enumerate}
  \item $\tyPath{\locEl{\locEl{R}{f}}{g}}{\locEl{R}{fg}}$
  \item $\tyPath{\locEl{R}{f}}{R}$, if $f~\func{∈}~R^\times$
  \item $\tyPath{\locEl{R}{f}}{\locEl{R}{g}}$,
    if $\nicefrac{f}{1}~\func{∈}~\locEl{R}{g}^\times$
    and $\nicefrac{g}{1}~\func{∈}~\locEl{R}{f}^\times$
  \end{enumerate}
\end{lemma}

\subsection{The Zariski lattice}
\label{ZarLat}

Next, we provide a definition of the Zariski lattice that does not lead to size issues,
while still being convenient to work with.
We have already seen that the Zariski lattice \ZL, which classically corresponds
to the compact open sets of the Zariski topology, can be described as the lattice
of radicals of finitely generated ideals. The meet and join of this lattice are defined
using multiplication and addition of ideals. With some elementary ideal theory
this should be straightforward to formalize.
Unfortunately, without any form of impredicativity,
like resizing axioms,
this leads to size issues.

So far we have avoided being explicit about universe levels, but in this section let $\ell$ be the level of the
base ring $R$, that is, the level of the universe in which the underlying type of $R$ lives.
Being precise about universe levels, subsets of $R$ are elements of $R\to\hProp~\ell$,
living in $\Type~(\ell+1)$, the next bigger universe.
The type of all ideals of $R$,
which is just the $\Sigma$-type of
subsets satisfying the ideal property, is hence in $\Type~(\ell+1)$. However, for technical reasons
to be discussed in the next section, we need $\ZL:\Type~\ell$.
Consequently, the definition of \ZL\ must not rely on the type of all ideals of $R$.

To avoid this issue, we use a construction due to Espa\~{n}ol \cite{Espanol83}.
Since we are only concerned with the radicals of finitely generated ideals,
we can describe \ZL\ in terms of generators instead of arbitrary ideals. In particular,
a list of generators $\alpha = [\alpha_0,\dots,\alpha_n]$ with $\alpha_i:R$ corresponds
to the radical of the ideal generated by the $\alpha_i$. In other words, we can obtain
\ZL\ by quotienting the type of lists with elements in $R$, by the relation
\[
  \alpha \sim \beta \quad\Leftrightarrow\quad
  \big(\forall~i\to\beta_i\in\sqrt{\langle\alpha_0,\dots,\alpha_n\rangle}\,\big)\,\, \text{and} \,\,
\big(\forall~i\to\alpha_i\in\sqrt{\langle\beta_0,\dots,\beta_m\rangle} \,\big)
\]
Here $\langle\alpha_0,\dots,\alpha_n\rangle$ is the ideal generated by the $\alpha_i$'s.
As both the type of lists and $\sim$ live in $\Type~\ell$ so does their quotient \ZL.
It might seem more natural to quotient by the relation
\begin{align*}
  \alpha \threesim \beta \quad\Leftrightarrow\quad
  \tyPath{\sqrt{\langle\alpha_0,\dots,\alpha_n\rangle}}{\sqrt{\langle\beta_0,\dots,\beta_m\rangle}}
\end{align*}
Unfortunately the type of paths between
two such radicals is large, as for any two ideals $I,J$ we have $\tyPath{I}{J}:\Type~(\ell+1)$.
Still, $\sim$ is equivalent to
$\threesim$ in the sense that we have
$\alpha\sim\beta$ if and only if $\alpha\threesim\beta$. This equivalence can then be used in proofs.

Equipping \ZL{} with the distributive lattice structure requires us to introduce operations
on lists that correspond to ideal addition and multiplication.
For the join we can take list concatenation $\plusplus{\_}{\_}$
as this corresponds
to addition of finitely generated ideals in the sense that for any two lists
$\alpha,\beta$ we have that
\begin{align}\label{plusplusfact}
  \begin{split}
  \langle\;\plusplus{[\alpha_0,\dots,\alpha_n]}{~[\beta_0,\dots,\beta_m]}\;\rangle
  ~&\func{≡}~\langle\alpha_0,\dots,\alpha_n,\beta_0,\dots,\beta_m\rangle \\
  ~&\func{≡}~\langle\alpha_0,\dots,\alpha_n\rangle+\langle\beta_0,\dots,\beta_m\rangle
    \end{split}
\end{align}
When checking that
$\plusplus{\_}{\_}$ defines an operation on the quotient \ZL, it suffices
to check that it respects $\threesim$, which in turn follows from (\ref{plusplusfact}).

For the meet of \ZL{} we need to define an operation $\_\cdot\cdot\_$ on lists
that corresponds to multiplication of finitely generated ideals.
For two lists $\alpha,\beta$ this product $\alpha\cdot\cdot\beta$ is the list
of all products of the form $\alpha_i\beta_j$.
Proving the correspondence to ideal multiplication, \ie\
\begin{align}\label{dotdotfact}
  \begin{split}
  \langle\;[\alpha_0,\dots,\alpha_n]\cdot\cdot\,[\beta_0,\dots,\beta_m]\;\rangle
  ~&\func{≡}~\langle\alpha_0\beta_0,\dots,\alpha_n\beta_0,\dots,\alpha_0\beta_m,\dots,\alpha_n\beta_m\rangle \\
  ~&\func{≡}~\langle\alpha_0,\dots,\alpha_n\rangle\cdot\langle\beta_0,\dots,\beta_m\rangle
    \end{split}
\end{align}
is much more involved than (\ref{plusplusfact}),
but gives us the well-definedness of $\_\cdot\cdot\_$
on the quotient. Proving the lattice laws also proceeds by using
(\ref{plusplusfact}) and (\ref{dotdotfact}), together with the equivalence of
$\sim$ and $\threesim$, thus reducing these laws to special cases of standard
equalities about ideal addition/multiplication and radical ideals.

Showing the universal property of \ZL{} is then relatively straightforward.
Note that the basic opens are defined by the map $D:R\to\ZL$,
sending $f:R$ to $\tmTrunc{[f]}$, the equivalence class of the singleton list $[f]$.
It then becomes straightforward to verify that for $f,g : R$
\begin{align*}
  D(g)\leq D(f)\;\Leftrightarrow\;\sqrt{\langle g\rangle}\subseteq\sqrt{\langle f\rangle} \;\Leftrightarrow\; f\in\locEl{R}{g}^\times\; \Leftrightarrow\;\func{isContr}\Big(\text{Hom}_R\big[\locEl{R}{f}\,,\,\locEl{R}{g}\,\big]\Big)
\end{align*}
The last two logical equivalences hold by some standard commutative algebra
and the universal property of localization.\footnote{As all the types above are
  propositions, we could also replace logical equivalence with
  equivalence of types \func{≃}.}
The basic opens as a subset of \ZL{} are defined as the function $\mathsf{BasicOpens}:\mathcal{L}_R\to\func{hProp}$,
sending $\mathfrak{a}$ to $\tyExists{f}{R}{\tyPath{D(f)}{\mathfrak{a}}}$.
In other words $\mathfrak{a}\in\mathsf{BasicOpens}$ if there merely exists an
$f$ such that $\mathfrak{a}$ equals $D(f)$. The type of basic opens is then
the type $\BO=\tySigma{\mathfrak{a}}{\ZL}{\mathfrak{a}\in\mathsf{BasicOpens}}$.
Note that by the universal property, the only lattice morphism
$\ZL\to\ZL$ commuting with $D$ is the identity and from this it follows that
for any list $\alpha=[\alpha_0,\dots,\alpha_n]$ the equivalence class $\tmTrunc{\alpha}$
is the finite join $\bigvee_{i=0}^n D(\alpha_i)$.
Since being a basis is a proposition, this is enough to prove
that the basic opens form a basis of \ZL.

\section{Category theory}
\label{CatTheory}

We now turn to category theory and describe the machinery needed to lift sheaves
from the basis of a distributive lattice to the whole lattice. The lifting of a presheaf defined
on a subset of a distributive lattice, seen as a sub-poset category,
is obtained by taking the right Kan extension along the inclusion.
The general theory of limits and Kan extensions  in the formalization
closely follows Mac Lane \cite{MacLaneCategories}.
We will not discuss details here, but only sketch the lattice case in order to introduce
notation and show where size issues enter the picture.

Note that for any category $\mathcal{C}$ and $P:\mathcal{C}\to\func{hProp}$,
$\mathcal{C}_P=\tySigma{x}{\mathcal{C}}{x \in P}$ becomes a subcategory of $\mathcal{C}$
by taking arrows between pairs to be arrows between the first projections.
The projection $\field{fst}$ induces a fully faithful embedding of
$\mathcal{C}_P$ into $\mathcal{C}$.
Let us now fix a distributive lattice $L : \Type~\ell$.
For any $P:L\to\hProp~\ell$, $L_P$ becomes a sub-poset of $L$.

Let $\mathcal{C}$ be an $\ell$-complete category
(\ie\ with limits of diagrams in $\Type~\ell$).
The right Kan extension then exists for any $\mathcal{C}$-valued presheaf $\mathcal{G}$
on $L_P$:
\[
\begin{tikzcd}
  \big(L_P\big)^{op}\arrow[d,hook,"\field{fst}"'] \arrow[rrd, "\mathcal{G}", end anchor={[yshift=0.3em]}]&&\\
  L^{op}\arrow[rr,dashed,"\mathsf{Ran}~\mathcal{G}"'] &&\mathcal{C}
\end{tikzcd}
\quad \big(\mathsf{Ran}~\mathcal{G}\big)\,(x) \;=\;\lim_{\longleftarrow}~\big\{\, \mathcal{G}(u) \to \mathcal{G}(v)~\vert~u,v:L_P\text{ s.t. } v\leq u\leq x\,\big\}
\]
Moreover, since the functor induced by \field{fst} is fully faithful,
$\mathsf{Ran}~\mathcal{G}$ extends $\mathcal{G}$ in the sense that we have a
natural isomorphism between $\mathcal{G}$ and $\big(\mathsf{Ran}~\mathcal{G})\circ\field{fst}$.
For the structure sheaf we need to consider presheaves valued in $\func{CommRing}~\ell$,
the category of commutative rings living in the same universe as the base ring $R$.
This category is  $\ell$-complete but \emph{not} $(\ell+1)$-complete.
It is precisely for this reason that  we required \ZL{} to be in $\Type~\ell$.

The main result of this section is that taking the right Kan extension of a
presheaf defined on the \emph{basis} of a lattice preserves
the sheaf property.\footnote{In fact the right Kan extension (as opposed to left Kan)
  establishes an equivalence of categories
  between sheaves on a lattice $L$ and sheaves on a basis $B$ of $L$, with
  its inverse being restriction to $B$. This is the special case of the
  so-called \emph{comparison lemma} for distributive lattices.}
This requires a definition of sheaf
on both distributive lattices and their bases suitable for formalization.
For the remainder of this section we fix a basis $B$ of $L$.
When outlining the formalization, we defined sheaves
on lattices by restricting the usual definition in terms of equalizer diagrams
to finite covers. However, we can express these equalizers
as finite limits over diagrams of a certain shape.\footnote{See e.g.\ Mac Lane \cite[Thm. V.2.1]{MacLaneCategories}.}
This approach is also taken by Coquand, Lombardi and Schuster in \cite{ConstrSchemes}.
We decided to follow it as it allows one to work
with special data types for the shapes of the diagrams involved, which is convenient in the formalization.

\begin{definition}[Sheaf diagram shapes]\label{DLShfDiagCat}
  The category of the sheaf diagram shape for covers of size $n$, has as objects
  indices $i$, where $1\leq i \leq n$, or pairs of indices $(i,j)$, where
  $1\leq i<j\leq n$. Arrows are either identity arrows or inclusions
  of singleton indices from the left $i\mapsto (i,j)$ or right $j\mapsto (i,j)$.
\end{definition}
In \Agda the objects and arrows can be described as the terms of the following
data types:
\ExecuteMetaData[agda/latex/CategoryTheoryCode.tex]{DLShfDiagOb}
\vspace{-.5cm}
\ExecuteMetaData[agda/latex/CategoryTheoryCode.tex]{DLShfDiagHom}
Here \func{Fin} \var{n} is the finite type of $n$ elements
from $1$ to $n$. Composition
is easily defined by case analysis as it is not possible to
compose two non-identity arrows and the laws
then follow directly. We denote the resulting category by \func{DLShfDiagCat} \var{n}.

\begin{remark}
  In order for this to define a category in HoTT/UF
  we have to prove that the hom-types are sets, \ie\ that for $x,y:\func{DLShfDiagOb}~n$ we have
  $\func{isSet}~(\func{DLShfDiagHom}~n~x~y)$.
  This follows from a retraction argument using the encode-decode method \cite{HoTTBook}.
\end{remark}
Given a list of elements $\alpha = [\alpha_1,\dots,\alpha_{n}]$ with $\alpha_i : L$, we get a corresponding
diagram in the form of a functor $\func{DLShfDiagCat}~n\to L^{op}$
sending the singleton index $i$ to $\alpha_i$ and $(i,j)$
to $\alpha_i\wedge \alpha_j$. We call this the \emph{diagram associated to
  $\alpha$}.
  Furthermore, let $\mathcal{F}:L^{op}\to\mathcal{C}$ be a presheaf,
  we then have a diagram $\func{DLShfDiagCat}~n\to\mathcal{C}$,  obtained
  by composing the diagram associated to $\alpha$ with $\mathcal{F}$.
  We call this the \emph{$\mathcal{F}$-diagram associated to $\alpha$}.

The join $\bigvee_{i=1}^{n}\alpha_i$ induces a cone over the
diagram associated to $\alpha$ and it is in fact a limiting cone because limits are least upper bounds in
the opposite of a poset category.
A presheaf on $L$ is a sheaf if it preserves these limits:

\begin{definition}[Sheaves on a distributive lattice]\label{SheafProp}
  We say that $\mathcal{F}$ is a sheaf on the distributive lattice $L$,
  if for all lists $\alpha=[\alpha_1,\dots,\alpha_{n}]$ with $\alpha_i:L$
  the induced cone of $\mathcal{F}\big(\bigvee_{i=1}^{n}\alpha_i\big)$ over
  the $\mathcal{F}$-diagram associated to $\alpha$
  is a limiting cone.
  In other words $\mathcal{F}\big(\bigvee_{i=1}^{n}\alpha_i\big)$ is the
  limit of the diagram
  \[
  \begin{tikzcd}
    &\mathcal{F}\big(\bigvee_{i=1}^{n}\alpha_i\big)\ar[dl]\ar[d]\ar[dr]&\\
    \mathcal{F}(\alpha_i)\ar[r]&\mathcal{F}(\alpha_i\wedge\alpha_j)&\mathcal{F}(\alpha_j)\ar[l]
  \end{tikzcd}\quad\quad\quad\text{for all}\quad1\leq i<j\leq n.
  \]
\end{definition}
We now turn our attention to the corresponding notion for the basis
$B$. Let $\mathcal{G}:B^{op}\to\mathcal{C}$ be a presheaf. For a list
$\alpha=[\alpha_1,\dots,\alpha_{n}]$ with $\alpha_i : B$, we have a
diagram $\func{DLShfDiagCat}~n\to\mathcal{C}$, which is obtained by
composing the diagram associated to $\alpha$ with $\mathcal{G}$.  We
call this the \emph{$\mathcal{G}$-diagram associated to $\alpha$}.
As $B$ is in general not closed under finite joins, the definition of a
basis-sheaf below has an extra condition, saying that limits of the
associated diagrams are only preserved if they exist.

\begin{definition}[Sheaves on a basis of a distributive lattice]\label{SheafPropBasis}
  We say that $\mathcal{G}$ is a sheaf on the basis B of a distributive lattice,
  if for all $\alpha=[\alpha_1,\dots,\alpha_{n}]$ with $\alpha_i : B$,
  such that $\bigvee_{i=1}^{n}\alpha_i$ is in $B$,
  the induced cone of $\mathcal{G}\big(\bigvee_{i=1}^{n}\alpha_i\big)$ over the
  $\mathcal{G}$-diagram associated to $\alpha$ is a limiting cone.
\end{definition}
The following lemma only holds for sheaves on the whole lattice,
since it requires closure under finite joins.
\begin{lemma}\label{PBreduce}
  Let $\mathcal{F}:L^{op}\to\mathcal{C}$, then $\mathcal{F}$ is sheaf
  if and only if $\mathcal{F}(\bot)$ is terminal in $\mathcal{C}$ and
  for all $x,y:L$ the following is a pullback square
  \[
    \begin{tikzcd}
      \mathcal{F}(x\vee y) \arrow[r]\pbsign{dr}\arrow[d] &\mathcal{F}(x) \arrow[d] \\
      \mathcal{F}(y) \arrow[r] &\mathcal{F}(x\wedge y)
    \end{tikzcd}
  \]
\end{lemma}
\begin{proof}
  We start by observing that \cref{SheafProp} also applies
  to the empty list $[]$. The join over $[]$
  is just $\bot$ and the associated diagram
  is the ``empty'' diagram. So if $\mathcal{F}$ is a sheaf then
  $\mathcal{F}(\bot)$ is terminal. Furthermore, the pullback squares are exactly the sheaf condition
  for two element lists. This concludes the ``only if'' direction.

  For the other direction, we proceed by
  induction on the length $n$. The base case $n=0$ follows from
  $\mathcal{F}(\bot)$ being terminal. For the inductive step
  take a list $\alpha_1,\dots,\alpha_{n}:L$ of length $n$.
  By assumption the following is a pullback square
  \[
    \begin{tikzcd}
      \mathcal{F}\big(\bigvee_{i=1}^{n}\alpha_i\big) \arrow[r]\pbsign{dr}\arrow[d] &\mathcal{F}\big(\bigvee_{i=2}^{n}\alpha_i\big) \arrow[d] \\
      \mathcal{F}(\alpha_1) \arrow[r] &\mathcal{F}\big(\bigvee_{i=2}^{n}(\alpha_1\wedge\alpha_i)\big)
    \end{tikzcd}
    \]
    Now both lists  $\alpha_1,\dots,\alpha_{n}$
    and $\alpha_1\wedge\alpha_1,\dots,\alpha_1\wedge\alpha_{n}$ are of length $n-1$.
    By applying the induction hypothesis to both, one can
    easily check that $\mathcal{F}\big(\bigvee_{i=1}^{n}\alpha_i\big)$ is the
    desired limit.
  \end{proof}
This alternative characterization can be used to prove our ``comparison lemma''
for distributive lattices.
For the remainder of this section, let $\mathcal{G}:B^{op}\to\mathcal{C}$ be a sheaf on the basis $B$.
The key observation is the following technical lemma.
\begin{lemma}\label{coverLemma}
  For any list of elements $\alpha_1,\dots,\alpha_{k}:B$,
  we have that\footnote{This is actually how the extension $\big(\mathsf{Ran}~\mathcal{G}\big)$ is defined in
    \cite{ConstrSchemes}. However, in general we cannot use concrete
    covers of arbitrary elements of $L$ by base elements to construct a functor into
    $\mathcal{C}$ if its h-level is unknown.}
  \begin{align}
    \big(\mathsf{Ran}~\mathcal{G}\big)\big(\textstyle\bigvee_{i=1}^{k}\alpha_i\big)\;\cong\;
    \displaystyle\lim_{\longleftarrow}~\big\{\, \mathcal{G}(\alpha_i) \to \mathcal{G}(\alpha_i\wedge\alpha_j)\leftarrow\mathcal{G}(\alpha_j)~\vert~1\leq i<j\leq k\,\big\}
  \end{align}
\end{lemma}
\begin{proof}[Proof sketch]
  By definition we have
  \begin{align*}
    \big(\mathsf{Ran}~\mathcal{G}\big)\big(\textstyle\bigvee_{i=1}^{k}\alpha_i\big)\;=\;\displaystyle\lim_{\longleftarrow}~\big\{\, \mathcal{G}(u) \to \mathcal{G}(v)~\vert~u,v:B\text{ s.t.\ } v\leq u\leq \textstyle\bigvee_{i=1}^{k}\alpha_i\,\big\}
  \end{align*}
  This immediately gives us the map from left to right,
  since we can restrict the defining diagram of
  $\big(\mathsf{Ran}~\mathcal{G}\big)\big(\textstyle\bigvee_{i=1}^{k}\alpha_i\big)$ to
  the $\mathcal{G}$-diagram associated to $\alpha$.

  For the inverse map we have to show that given any $X:\mathcal{C}$ with a cone
  based at $X$ over the $\mathcal{G}$-diagram associated to $\alpha$,
  we can extend this to a cone based at $X$ over the defining diagram of
  $\big(\mathsf{Ran}~\mathcal{G}\big)\big(\textstyle\bigvee_{i=1}^{k}\alpha_i\big)$.
  Assume we have $X:\mathcal{C}$ with such a cone  and
  let $u:B$ such that $u\leq\bigvee_{i=1}^{k}\alpha_i$.
  Then $\tyPath{\textstyle\bigvee_{i=1}^{k}(u\wedge\alpha_i)}{u}$ and
  hence $\textstyle\bigvee_{i=1}^{k}(u\wedge\alpha_i)$ is in $B$.
  This means that we can apply the assumption that $\mathcal{G}$ is a sheaf to this join.
  By substituting along this path, we can see $\mathcal{G}(u)$
  as the limit of the $\mathcal{G}$-diagram associated to the $u\wedge\alpha_i$'s.
  By composing with restrictions we get a cone based at $X$ over
  the $\mathcal{G}$-diagram associated to the $u\wedge\alpha_i$'s, and thus
  an arrow $X\to\mathcal{G}(u)$. It is not hard to show that this is functorial
  in $u$, which gives us the desired inverse arrow.
  The proof that the two maps are mutually inverse, is
  quite cumbersome and we will omit it here.
\end{proof}
The proof of the following theorem is the most technical of
the entire formalization, so again we only give an outline.

\begin{theorem}\label{lemma1}
  $\mathsf{Ran}~\mathcal{G}$ is a sheaf on the distributive lattice $L$.
\end{theorem}
\begin{proof}[Proof sketch]
  It suffices to check the terminal and pullback condition of \cref{PBreduce}.
  We will restrict our attention to the pullback case here. Let $x,y:L$ and
  note that, as being a pullback square is a proposition,
  we can take covers $\tyPath{x}{\bigvee_{i=1}^{n}\beta_i}$
  and $\tyPath{y}{\bigvee_{i=1}^{m}\gamma_i}$
  by base elements, i.e.\ $\beta_i,\gamma_j: B$  for all $i$ and $j$. Substituting
  these covers for $x$ and $y$,
  we have to prove the following:
  given $X:\mathcal{C}$ and arrows $f$ and $g$ such that the outer square in the diagram below
  commutes, then there is a unique arrow $h$ making the whole diagram commute:
  \begin{equation}\label{bigDiag}
  \begin{tikzcd}
    X \arrow[dr, dashed, "\exists!~h"]\arrow[drr,bend left=20, "f"]\arrow[ddr,bend right, "g"']&&\\
      &\big(\mathsf{Ran}~\mathcal{G}\big)\Big(\bigvee_{i=1}^{n+m}(\plusplus{\beta}{\gamma})_i\Big) \arrow[r]\arrow[d] &\big(\mathsf{Ran}~\mathcal{G}\big)\Big(\bigvee_{i=1}^{n}\beta_i\Big) \arrow[d] \\
      &\big(\mathsf{Ran}~\mathcal{G}\big)\Big(\bigvee_{i=1}^{m}\gamma_i\Big) \arrow[r] &\big(\mathsf{Ran}~\mathcal{G}\big)\Big(\big(\bigvee_{i=1}^{n}\beta_i\big)\wedge \big(\bigvee_{i=1}^{m}\gamma_i\big)\Big)
    \end{tikzcd}
  \end{equation}
  Here $(\plusplus{\beta}{\gamma})$ is the list-concatenation of $\beta$ and $\gamma$.
  Applying \cref{coverLemma} to $(\plusplus{\beta}{\gamma})$, we get such an arrow $h$
  from a cone based at $X$ over the diagram
  \[
  \big\{\, \mathcal{G}\big((\plusplus{\beta}{\gamma})_i\big) \to \mathcal{G}\big((\plusplus{\beta}{\gamma})_i\wedge(\plusplus{\beta}{\gamma})_j\big)\leftarrow\mathcal{G}\big((\plusplus{\beta}{\gamma})_j\big)~\vert~1\leq i<j\leq n+m\,\big\}
  \]
  To construct such a cone, we apply \cref{coverLemma} to both $\beta$ and $\gamma$
  and precompose the resulting limiting cones with $f$ and $g$ respectively.
  This gives us two cones based at $X$, one over the
  $\mathcal{G}$-diagram associated to $\beta$
  and the other one over the
  $\mathcal{G}$-diagram associated to $\gamma$.
  Note that the two cones are compatible in the following sense:
  for all $1\leq i\leq n$ and $1\leq j\leq m$ the following square commutes
  \[
  \begin{tikzcd}
      X \arrow[r]\arrow[d] &\mathcal{G}(\beta_i) \arrow[d] \\
      \mathcal{G}(\gamma_j) \arrow[r] &\mathcal{G}(\beta_i\wedge \gamma_j)
  \end{tikzcd}
  \]
  This is because the outer square in diagram (\ref{bigDiag}) commutes
  and it is sufficient to construct a cone based at $X$ over the
  $\mathcal{G}$-diagram associated to $(\plusplus{\beta}{\gamma})$.

  Note that the induced $h$ is the unique cone morphism between the cone thus constructed
  and the limiting cone obtained from applying \cref{coverLemma} to
  $(\plusplus{\beta}{\gamma})$. Moreover, $f$ and $g$ are the unique cone morphisms
  between their respective precomposition-cones based at $X$ and the limiting cones
  obtained from applying \cref{coverLemma} to $\beta$ and $\gamma$ respectively.
  From this it follows by a cumbersome diagram chase that $h$ is the unique morphism
  making the two triangles in diagram (\ref{bigDiag}) commute.
\end{proof}
Formalizing the gaps in the above proof sketches is quite tedious and
uses involved transports. We refer the interested reader to the
formalization.

\section{The structure sheaf}
\label{sec:structuresheaf}

We now have all the ingredients needed to formalize the structure sheaf.
The basic opens \BO~form a basis of $\mathcal{L}_R$
and we have seen in the previous section how sheaves can be extended along the embedding
$\field{fst}:\mathcal{B}_R\to \mathcal{L}_R$.
What should the structure sheaf on $\mathcal{B}_R$ then look like?
Focusing on the underlying presheaf and its action on objects for now,
we need a function $\mathcal{B}_R\to\func{CommRing}~\ell$,
which upon unfolding the definition of $\mathcal{B}_R$ becomes
\begin{align*}
  \big(\tySigmaNoParen{\mathfrak{a}}{\mathcal{L}_R}
          {\underbrace{\tyExists{f}{R}{\tyPath{D(f)}{\mathfrak{a}}}}_{\text{prop. trunc.}}}\big)
          \;\longrightarrow\;\underbrace{\func{CommRing}~\ell}_{\text{groupoid}}
\end{align*}
Since membership in $\mathcal{B}_R$ is defined as a mere existence condition
using propositional truncation, we can only specify the behavior of the structure sheaf
in the case where we are given a point constructor of this truncation.
If $\mathfrak{a}:\ZL$ is a basic open, such an element of the truncation
consists of an element $f:R$ and a path $p:\tyPath{D(f)}{\mathfrak{a}}$.
In this case we know that the structure sheaf should
send $(\,\tmPair{\mathfrak{a}}{\tmPtrunc{\tmPair{f}{p}}}\,)$ to $\locEl{R}{f}$.
If the goal type were a proposition, this would be enough to specify a function.
However, the type of commutative rings
is a groupoid, requiring us to construct some non-trivial higher coherences.

To circumvent this problem we use the observation that the localizations
are actually $R$-algebras and that we could regard the structure sheaf as taking values
in $R$-algebras. What is usually called the structure sheaf in the literature is this
$R$-algebra-valued sheaf composed with the forgetful functor to commutative rings.
In other words, the structure sheaf factors through the forgetful functor from
$R$-algebras to commutative rings. The single reason why the situation is more well-behaved in
$R$-algebras is the fact that
\begin{align*}
  D(g)\leq D(f) \quad\Longleftrightarrow\quad\func{isContr}\Big(\text{Hom}_R\big[\locEl{R}{f}\,,\,\locEl{R}{g}\,\big]\Big)
\end{align*}
Contractibility is a powerful concept in HoTT/UF and we will show how this can be
used to solve the coherence issues of the structure sheaf and gives rise to a reduction
argument for the sheaf property.
We start with two lemmas for general constructions involving propositional
truncations and $R$-algebras. Note that these results are pretty much tailored
to the situation of the structure sheaf, but should also hold for other univalent categories,
which are always groupoids and even sets if they are posetal \cite[Lemma 9.1.9, Ex. 9.1.14]{HoTTBook}. With a bit of abuse of notation we will use $\RAlg$ to denote both
the type and the category of $R$-algebras.
\begin{lemma}\label{recPTtoCommAlgebra}
  Let $X:\Type$ and $\mathcal{F}:X\to\RAlg$.
  Assume further that for $x,y:X$ we have an isomorphism of $R$-algebras
  $\varphi_{xy}:\mathcal{F}(x)\cong\mathcal{F}(y)$ such that
  for $x,y,z:X$ we have a path $\tyPath{\varphi_{xz}}{\varphi_{yz}\circ\varphi_{xy}}$.
  Then we can construct a map $\norm{\mathcal{F}}: \ptrunc{X}\to\RAlg$
  such that for $x:X$ we have
  $\norm{\mathcal{F}}\big(\,\tmPtrunc{x}\,\big)=\mathcal{F}(x)$ definitionally.
\end{lemma}
\begin{proof}
  Since $\RAlg$ is an groupoid, we can
  apply a result by Kraus \cite[Prop. 2.3]{KrausPropTrunc}.
  In order to construct $\norm{\mathcal{F}}$ we need a family of paths
  over any two elements of $X$ satisfying a certain coherence condition.
  For $x,y:X$ we get a path $\func{sip}~\varphi_{xy}:\tyPath{x}{y}$.
  The corresponding coherence condition states that for $x,y,z:X$,
  we need a path
  $\tyPath{\func{sip}~\varphi_{xz}}{\func{sip}~\varphi_{xy}~\func{∙}~\func{sip}~\varphi_{yz}}$ (were \func{\_∙\_} is path composition).
  By the functoriality of \func{sip},
  which follows from the functoriality of \func{ua},
  this path type is equivalent to
  $\tyPath{\func{sip}~\varphi_{xz}}{\func{sip}~(\varphi_{yz}\circ\varphi_{xy})}$.
  But by assumption we have
  $\tyPath{\varphi_{xz}}{\varphi_{yz}\circ\varphi_{xy}}$, so by applying \func{sip}
  to this path
  we are done.
\end{proof}
For the next lemma, note that for any category $\mathcal{C}$ and
family $P:\mathcal{C}\to\Type$, we have the subcategory $\mathcal{C}_{\norm{P}}$
of $\mathcal{C}$ induced by $\lambda~x\to\ptrunc{P(x)}:\mathcal{C}\to\hProp$.
\begin{lemma}\label{universalPShf}
  Let $\mathcal{C}$ be a category with a family $P:\mathcal{C}\to\Type$
  and a family of $R$-algebras
  $\mathcal{F}:\big(\tySigmaNoParen{x}{\mathcal{C}}{P(x)}\big)\to \RAlg$.
  Assume furthermore that for $x,y:\mathcal{C}$, $p:P(x)$, $q:P(y)$
  with an arrow $f:\mathcal{C}\,[x,y]$ we have
  \begin{align*}
    \func{isContr}\Big(\text{Hom}_R\big[\mathcal{F}(\tmPair{y}{q})\,,\,\mathcal{F}(\tmPair{x}{p})\,\big]\Big)
  \end{align*}
  We can then construct a ``universal'' presheaf
  \begin{align*}
    \universalPShf:\big(\mathcal{C}_{\norm{P}}\big)^{op}\to \RAlg
  \end{align*}
  such that for  $x:\mathcal{C}$ with $p:P(x)$ we have
  \begin{align*}
    \universalPShf(\tmPair{x}{\tmPtrunc{p}})=\mathcal{F}(\tmPair{x}{p})
  \end{align*}
  and for $y:\mathcal{C}$, $q:P(y)$ with arrow $f:\mathcal{C}\,[x,y]$,
  $\universalPShf\,(f)$ is the unique  $R$-algebra morphism
  from $\mathcal{F}(\tmPair{y}{q})$ to $\mathcal{F}(\tmPair{x}{p})$.
\end{lemma}
\begin{proof}
  We first describe the action of \universalPShf~on objects.
  By currying we fix $x:\mathcal{C}$ and need to provide a function
  $\ptrunc{P(x)}\to\RAlg$. For this we apply \cref{recPTtoCommAlgebra}
  to $\mathcal{F}(\tmPair{x}{\_}):P(x)\to\RAlg$.
  From our contractibility assumption it follows that given $p,q:P(x)$
  there are unique morphisms from $\mathcal{F}(\tmPair{x}{p})$ to
  $\mathcal{F}(\tmPair{x}{q})$ and vice versa,
  so
  $\mathcal{F}(\tmPair{x}{p})\cong \mathcal{F}(\tmPair{x}{q})$.
  It remains to check that the family of isomorphisms thus defined
  is closed under composition in the sense of \cref{recPTtoCommAlgebra}.
  Again, this follows from contractibility.

  For the action of \universalPShf~on morphisms, we start by proving something stronger.
  Given $x,y:\mathcal{C}$, $p:\ptrunc{P(x)}$, $q:\ptrunc{P(y)}$ with
  an arrow $f:\mathcal{C}[x,y]$, we have:
  \begin{align*}
    \func{isContr}~\Big(\text{Hom}_R\,\big[\,\universalPShf(\tmPair{x}{p})~,~\universalPShf(\tmPair{y}{q})\,\big]\Big)
  \end{align*}
  As being contractible is a proposition, we can assume that $p=\tmPtrunc{p'}$ and $q=\tmPtrunc{q'}$.
  In this case $\universalPShf(\tmPair{x}{p})=\mathcal{F}(\tmPair{y}{p'})$ and
  $\universalPShf(\tmPair{y}{q})=\mathcal{F}(\tmPair{y}{q'})$ and we can just use
  our contractibility hypothesis. Since a morphism between $(\tmPair{x}{p})$ and
  $(\tmPair{y}{q})$ in $\mathcal{C}_{\norm{P}}$
  is just a morphism $f:\mathcal{C}[x,y]$, we can take $\universalPShf(f)$
  to be the center of contraction of the contractible type of $R$-algebra morphisms above.
  The functoriality of \universalPShf~then follows immediately.
\end{proof}
We now want to apply this construction to the Zariski lattice (seen as a poset category).
In the situation of \cref{universalPShf} with $\mathcal{C}=\mathcal{L}_R$
we set, for $\mathfrak{a}:\mathcal{L}_R$:
\begin{align*}
  P(\mathfrak{a})\;=\; \tySigma{f}{R}{\tyPath{D(f)}{\mathfrak{a}}}
  \quad\quad\text{and}\quad\quad
  \mathcal{F}(\tmPair{\mathfrak{a}}{\tmPair{f}{p}})\;&=\;\locEl{R}{f}.
\end{align*}
If we're given $\mathfrak{b}\leq\mathfrak{a}$ with $\tyPath{D(f)}{\mathfrak{a}}$
and $\tyPath{D(g)}{\mathfrak{b}}$ then $D(g)\leq D(f)$ and the
type of $R$-algebra morphisms from $\locEl{R}{f}$ to $\locEl{R}{g}$ is contractible.
This way we obtain the desired
\begin{align*}
  \universalPShf:\big(\mathcal{B}_R\big)^{op}\to\RAlg
\end{align*}
Composing with the forgetful functor from $R$-algebras to commutative rings gives us the desired
presheaf on basic opens, denoted by \func{𝓞ᴮ}. From this we finally obtain the
structure (pre-)sheaf $\func{𝓞} : \big(\mathcal{L}_R\big) ^{op} \to \func{CommRing}$
using the right Kan extension machinery described in \cref{CatTheory}.
The following fact then becomes rather straightforward to verify:
\begin{proposition}\label{baseSections}
  For any $f:R$ we get a path $\func{𝓞}\big(D(f)\big)~\func{≡}~\locEl{R}{f}$.
\end{proposition}
\begin{proof}
  There is a canonical proof $p_f =\tmPtrunc{\tmPair{f}{\func{refl}}}$ of $D(f)$ belonging to the basic opens. Since we have a natural isomorphism
  between $\func{𝓞ᴮ}$ and $\func{𝓞}\circ\field{fst}$,
  we can use the SIP for commutative rings to obtain a path
  $\tyPath{\func{𝓞}\big(D(f)\big)}{\func{𝓞ᴮ}\big(\tmPair{D(f)}{p_f}\big)}$.
  But in $R$-algebras $\universalPShf\big(\tmPair{D(f)}{p_f}\big)$ equals $\locEl{R}{f}$
  definitionally and applying the forgetful functor to this gives us $\locEl{R}{f}$
  as a commutative ring (unfortunately not by \func{refl}).
\end{proof}
As a corollary we obtain the standard sanity check:
\begin{corollary}\label{globalSections}
  $\func{𝓞}\big(\top_{\mathcal{L}_R}\big)~\func{≡}~\func{𝓞}\big(D(1)\big)~\func{≡}~R$.
\end{corollary}
\begin{proof}
  $D(1)$ is the top element of the Zariski lattice by definition, so the first path is just
  \func{refl}. By \cref{baseSections} we get that
  $\func{𝓞}\big(D(1)\big)~\func{≡}~\locEl{R}{1}$. Combining this with \cref{InvElIds}.2,
  we get the desired path.
\end{proof}
It remains to prove that \func{𝓞ᴮ} is indeed a sheaf.
At this point the standard strategy is to reduce the general case of a cover
$\tyPath{D(h)}{\bigvee_{i=1}^n D(f_i)}$ to the special case $h=1$
and then proceed by some algebraic computations in the rings $\locEl{R}{f_i}$.\footnote{See for example
  \cite[theorem 2.33.]{GoertzWedhorn}, \cite[theorem 1.3.7]{EGA1} or \cite[theorem V.3.3]{StoneSpaces}. Note that in these classical textbooks the sheaf property only has
  to be verified for finite covers because basic opens are quasi-compact. In contrast,
  we are restricted to finite covers by definition.}
Informally this reduction step follows from a short argument, but it identifies
certain localizations by appealing to their canonical isomorphisms.
Making this formal in a system without univalence requires to take the isomorphisms
at face value and results in cumbersome diagram chases.
This problem is described in detail in \cite{SchemesLean}. There the ultimate breaking point was
identifying the rings $\locEl{\locEl{R}{f}}{g}$ and $\locEl{R}{fg}$. As the authors point out,
simply providing a path between those rings won't solve the problem at hand,
since what is actually needed is a path between the diagrams occurring in the sheaf condition.
For the remainder of this section we want to show that
we can conclude that \func{𝓞ᴮ} is a sheaf from the aforementioned special case,
using the observation that the canonical morphisms are unique in $R$-algebras.
In our formalization, the special case of covers of $D(1)$ reads as follows:

\begin{lemma}\label{oneIdealToLimLemma}
  For a ring $A$ with $f_1 ,\dots,f_n:A$ such that
  $1\in\langle f_1 ,\dots,f_n\rangle$, we have
  \begin{align*}
    A~\func{≡}~\lim_{\longleftarrow}~\big\{\, \locEl{A}{f_i} \to \locEl{A}{f_if_j} \leftarrow \locEl{A}{f_j}~\vert~ 1\leq i < j \leq n\,\big\}
  \end{align*}
More precisely, the canonical cone of $A$ over the diagram above is a limiting cone.
\end{lemma}
\begin{proof}
  The proof follows closely the textbook approach,
  see \eg\ Mac Lane and Moerdijk~\cite[p. 125]{SheavesInGeometryAndLogic},
  by some hands-on algebra in the different rings involved.
  It is precisely at this point that working with concrete implementations of
  the $\locEl{A}{f_i}$ as set quotients really simplifies the formalization.
\end{proof}
Reducing the sheaf property of \func{𝓞ᴮ} to \cref{oneIdealToLimLemma} can now be
done using the special nature of \universalPShf.
We also need that the forgetful functor preserves and reflects limits
and some basic results about dependent paths.
In the library this is packaged up in a generalized, technical lemma,
working for arbitrary diagrams, not only those needed for the sheaf property.
For the sake of readability however, we proceed to prove our main result directly.

\begin{theorem}\label{mainThm}
  \func{𝓞ᴮ} is a sheaf on the basic opens.
\end{theorem}
\begin{proof}
  Again for readability, we restrict ourselves to the case of binary covers,
  \ie\ the situation where $\tyPath{D(h)}{D(f)\vee D(g)}$ for $f,g,h:R$.
  As described in the proof of \cref{PBreduce},
  in this case the sheaf property can be reformulated as stating
  that \textsf{sq} below is a pullback.
  \[
      \begin{tikzcd}
        \func{𝓞ᴮ}\big(\tmPair{D(h)}{p_h}\big) \arrow[r]\arrow[d]\arrow[dr,phantom,"\mathsf{sq}" description] &\func{𝓞ᴮ}\big(\tmPair{D(g)}{p_g}\big)\arrow[d] \\
        \func{𝓞ᴮ}\big(\tmPair{D(f)}{p_f}\big) \arrow[r] &\func{𝓞ᴮ}\big(\tmPair{D(fg)}{p_{fg}}\big)
      \end{tikzcd} \\
    \quad\quad
      \begin{tikzcd}
        \locEl{R}{h} \arrow[r]\arrow[d]\arrow[dr,phantom,"\mathsf{sq}_R" description] &\locEl{R}{g}\arrow[d] \\
        \locEl{R}{f} \arrow[r] &\locEl{R}{fg}
      \end{tikzcd} \\
  \]
  Here the $p$'s are, as in the proof of \cref{baseSections},
  the canonical proofs that the $D$'s are in fact basic opens.
  Note that by definition, \textsf{sq} is obtained by applying the forgetful
  functor to $\mathsf{sq}_R$ and since the forgetful functor \emph{preserves limits}
  (and in particular pullbacks) it suffices to prove that $\mathsf{sq}_R$ is a pullback
  in $R$-algebras.

  The assumption $\tyPath{D(h)}{D(f)\vee D(g)}$ gives us
  $\tyPath{\sqrt{\langle h\rangle}}{\sqrt{\langle f,g\rangle}}$
  and by some standard algebra
  $1\in\langle\nicefrac{f}{1},\nicefrac{g}{1}\rangle$ in $\locEl{R}{h}$.
  This lets us apply \cref{oneIdealToLimLemma} with $A=\locEl{R}{h}$ and
  we get that $\mathsf{sq}^*$ is a pullback (in rings):
  \[
      \begin{tikzcd}
        \locEl{R}{h} \arrow[r]\arrow[d]\pbsign{dr}\arrow[dr,phantom,"\mathsf{sq}^*"] &\locEl{\locEl{R}{h}}{g}\arrow[d]\\
        \locEl{\locEl{R}{h}}{f} \arrow[r] &\locEl{\locEl{R}{h}}{fg}
      \end{tikzcd} \\
  \]
  As all the vertices of $\mathsf{sq}^*$ are $R$-algebras, by the canonical morphisms
  coming from $R$, and all the edges of $\mathsf{sq}^*$ commute with these canonical morphisms,
  we can lift $\mathsf{sq}^*$ to a square $\mathsf{sq}^*_R$ in $R$-algebras.
  Since the forgetful functor \emph{reflects limits} (and thus pullbacks), we get
  that $\mathsf{sq}^*_R$ is a pullback square as well.

  All that we need is a path $\tyPath{\mathsf{sq}^*_R}{\mathsf{sq}_R}$ and we are done,
  as we can transport \emph{the property of being a pullback square} along this path of squares.
  It is immediate in \CubicalAgda that to give a path between squares we need to give
  four paths between the respective vertices and four dependent paths between the morphisms
  over the paths of vertices.
  In order to see how this applies to our situation,
  let us first look at the left side of $\mathsf{sq}^*_R$ and $\mathsf{sq}_R$.
  We get the following square where we have to provide paths at the top and bottom
  and a dependent path filling this square connecting the vertical arrows
  $\psi$ and $\varphi$:
  \[
    \begin{tikzcd}
      \locEl{R}{h}\ar[draw=none]{r}[auto=false]{\func{≡}\!\func{≡}\!\func{≡}\!\func{≡}\!\func{≡}}\arrow[d,"\psi"'] & \locEl{R}{h}\arrow[d,"\varphi"] \\
      \locEl{\locEl{R}{h}}{f}\ar[draw=none]{r}[auto=false]{\func{≡}\!\func{≡}\!\func{≡}\!\func{≡}} & \locEl{R}{f}
    \end{tikzcd}
  \]
  For the top path we just choose \func{refl}.
  For the bottom we apply \cref{InvElIds} and get a path
  \begin{align*}
    \tyPath{\locEl{\locEl{R}{h}}{f}}{\tyPath{\locEl{R}{hf}}{\locEl{R}{f}}}
  \end{align*}
  where the first path is just \cref{InvElIds}.1 and the second path is \cref{InvElIds}.2
  using the fact that $\tyPath{D(hf)}{D(f)}$ by absorption.
  Let $p$ denote the composition of these two paths.
  The dependent path between $\psi$ and $\varphi$ is then of type
  \begin{align*}
    \tyPathP{i}{\text{Hom}_R\big[\locEl{R}{h}\,,\,p~i\,\big]}{\psi}{\varphi}
  \end{align*}
  By a standard result about \func{PathP}, this is equivalent to the non-dependent path type
  \begin{align*}
    \tyPath{\func{transport}~\Big(\lambda~i\to\text{Hom}_R\big[\locEl{R}{h}\,,\,p~i\,\big]\Big)~\psi}{\varphi}
  \end{align*}
  But by definition $\varphi$ is the center of contraction of the type
  $\text{Hom}_R\big[\locEl{R}{h}\,,\,\locEl{R}{f}\,\big]$. By contractibility, we hence
  get a path to the transport of $\psi$ and thus the desired dependent path.
  Repeating this strategy four times, as described in the diagram below,
  gives us the desired path $\tyPath{\mathsf{sq}^*_R}{\mathsf{sq}_R}$ and finishes the proof.
  \begin{center}
  \begin{tikzcd}
    \locEl{R}{h} \arrow[rrrr,"\func{∃!}"{name=ot}]\arrow[dddd,"\func{∃!}"'{name=ol}]
    \ar[draw=none]{dr}[sloped,auto=false]{\func{≡}\!\func{≡}\!\func{≡}\!\func{≡}\!\func{≡}}
    &&&& \locEl{R}{g} \arrow[dddd,"\func{∃!}"{name=or}]
         \ar[draw=none]{dl}[sloped,auto=false]{\func{≡}\!\func{≡}\!\func{≡}\!\func{≡}\!\func{≡}} \\
    & \locEl{R}{h} \arrow[rr,""'{name=it}]\arrow[dd,""{name=il}]\pbsign{ddrr}
    && \locEl{\locEl{R}{h}}{g} \arrow[dd,""'{name=ir}] & \\
    &&&& \\
    & \locEl{\locEl{R}{h}}{f} \arrow[rr,""{name=ib}] && \locEl{\locEl{R}{h}}{fg} &\\
    \locEl{R}{f} \arrow[rrrr,"\func{∃!}"'{name=ob}]
    \ar[draw=none]{ur}[sloped,auto=false]{\func{≡}\!\func{≡}\!\func{≡}\!\func{≡}\!\func{≡}}
    &&&& \locEl{R}{fg}
    \ar[draw=none]{ul}[sloped,auto=false]{\func{≡}\!\func{≡}\!\func{≡}\!\func{≡}\!\func{≡}}
    \arrow[Rightarrow,shorten <=5pt, shorten >=6pt,  from=ir, to=or, "\func{PathP}" description]
    \arrow[Rightarrow,shorten <=6pt, shorten >=6pt,  from=il, to=ol, "\func{PathP}" description]
    \arrow[Rightarrow,shorten <=3pt, shorten >=6pt,  from=it, to=ot, end anchor={[xshift=-1.2ex]}, "\func{PathP}" description]
    \arrow[Rightarrow,shorten <=6pt, shorten >=6pt,  from=ib, to=ob, end anchor={[xshift=-0.3ex]}, "\func{PathP}" description]
  \end{tikzcd}
  \end{center}
\end{proof}
Combining this with \cref{lemma1} we get:

\begin{corollary}
  \func{𝓞} is a sheaf on the Zariski lattice \ZL.
\end{corollary}
Most of the argument in the proof of \cref{mainThm}, including
the crucial transport goes through for the general \universalPShf~construction
and cones over \emph{arbitrary diagrams}. If we take the action
of \universalPShf~on any cone of any shape, we only need two things for establishing
that this is a limiting cone:
first, a limiting cone in $R$-algebras of the same shape
and second, a family of paths between the corresponding vertices of the two cones.
In the case of structure sheaf the limiting cone is provided by \cref{oneIdealToLimLemma}
and the paths are provided by \cref{InvElIds}.
As a matter of fact, the general case is actually easier to formalize and
computationally better behaved, even though the pullback case is easier to visualize.

\section{Conclusion}
In this paper we presented a fully constructive and predicative
formalization of the structure sheaf on the Zariski lattice
in \CubicalAgda. To this end, we gave a construction of the Zariski lattice associated to
a commutative ring that does not increase the universe level even when working
predicatively. We formalized the notion of sheaf on a distributive lattice
and formally proved the first steps towards
a ``comparison lemma'' for distributive lattices.
In particular, we showed how to extend a sheaf defined on the basis of a lattice,
and taking values in any complete category, to a sheaf on the whole lattice.
Applying this to the Zariski lattice we then constructed the structure sheaf
on its basis. We had to solve higher coherence conditions in order to show
that this construction is well-defined. The main insight was that by essentially
regarding the structure sheaf to be valued in algebras, not rings,
we could use contractibility to solve the coherence issues.
Furthermore, it was the same contractibility results that let us formalize
the textbook proof of the sheaf property with the help of some cubical machinery.

As discussed in the introduction nothing in the paper crucially relied
on cubical features, but they proved convenient in the formalization.
In particular, having more things holding by \func{refl}, eliminators
computing also for higher constructors, and having direct access to
dependent paths in the form of \func{PathP} types simplified many of
the formal proofs. We hope nevertheless that the main ideas introduced
in this paper could prove useful for formalizations in other systems.
For the remainder of this paper we want to make a few comments
that should help putting our work into context.

\subsection{Comparison to the classical definition of affine schemes}
\label{subsec:classicalcomparison}
Even though the constructive, predicative approach described in this paper
is similar to the standard, classical textbook approach to affine schemes
in the sense that it involves a ``lifting'' from basic opens,
it might not be immediately clear whether we loose anything by working
with the Zariski lattice and finitary lattice sheaves.
As mentioned in the introduction, from a classical perspective this is not the case
because $\Spec R$ is a \emph{coherent} space.
A topological space $X$ is coherent if it is \emph{compact}, \emph{sober}
(its non-empty irreducible closed subsets are the closure of a single point),
and its compact opens are closed under finite intersections and form a basis of the topology of $X$. A coherent map between coherent spaces
$X$ and $Y$ is a continuous map $f:X\to Y$ such that for any compact open $K\subseteq Y$,
its pre-image $f^{-1}(K)$ is compact as well. Stone's representation theorem for
distributive lattices \cite{StoneRepresentation} states that the functor
from the category of coherent spaces with coherent maps to distributive lattices,
sending a coherent space to the lattice of its compact opens, is an equivalence
of categories.\footnote{
  Furthermore, any coherent space is coherently homeomorphic to $\Spec R$
  for some ring $R$ \cite{CoherentSpaces}, i.e.\ $\Spec$ as a functor
  from commutative rings to coherent spaces is essentially surjective.}
For the inverse direction we take a distributive lattice and
recover the opens of the corresponding space by taking \emph{ideals} on that lattice.
We can even recover the points of the space by taking \emph{prime filters} on the lattice.
In the case of $\Spec R$ the prime filters of \ZL~are just the complements of prime ideals of $R$.\footnote{See also the discussion by Coquand,
  Lombardi and Schuster in the introduction of \cite{ProjSpec}.}

The approach of defining \ZL{} through formal generators $D(f)$ and
obtaining the locale of Zariski opens as the ideals of \ZL, is taken in Johnstone's
``Stone Spaces'' \cite[Chap. V.3]{StoneSpaces}.
The structure sheaf on the resulting locale of \ZL-ideals can then be constructed by only defining
it on the base elements $D(f)$. In our predicative and constructive setting we
only extend the structure sheaf construction on basic opens to \ZL.
Again, classically no information is lost. Whether one considers
the structure sheaf to be defined on $\Spec R$ as a topological space,
on the locale of \ZL-ideals or only on \ZL, it is determined (up to unique isomorphism)
by what happens at the level of basic opens.

More generally, for any coherent space $X$, the category of sheaves on $X$
is equivalent to the category of (finitary) lattice-sheaves on the compact opens of $X$.
This follows from the comparison lemma for topological spaces, which gives us an
equivalence between sheaves on $X$ and sheaves on the basis of compact opens of $X$.
But since the compact opens are all compact we only have to consider
finite covers for the sheaf property, which gives us the equivalence to
lattice-sheaves on compact opens. Formalizing this classical fact would certainly
be interesting in its own right. But as we are interested in the formalization of
constructive mathematics, we will just see this fact as a justification that
the notion of constructive affine scheme that we arrive at is not
fundamentally weaker than the standard classical definition.

\subsection{Existing formalizations}
To our knowledge, we have presented the first constructive and predicative formalization
of affine schemes. However, there are several classical
formalizations of affine and general schemes
in the literature by now. Examples include an early setoid-based formalization in \Coq
by Chicli \cite{SchemesCoq},
the aforementioned formalization in \Lean's \mathlib~\cite{SchemesLean},
a more recent formalization in \systemname{Isabelle/HOL} \cite{SchemesHOL},
and a univalent \Coq formalization in the \UniMath library~\cite{SchemesUnimath}.
It is noteworthy that none of these formalizations define the structure sheaf
on basic opens first. Instead, they follow
the approach of Hartshorne's classic textbook ``Algebraic Geometry'' \cite{Hartshorne}.
This approach directly defines the structure sheaf on arbitrary opens, but
is inherently non-constructive. Assuming classical reasoning
(including the axiom of choice) it is quite straightforward
to formalize Hartshorne's definition. As a result, the \UniMath
formalization \cite{SchemesUnimath} does not actually use univalence in its
definition of the structure sheaf.

It should be mentioned however, that in the beginning the \Lean
formalization \cite{SchemesLean} did use the ``lift from basic opens approach''.
Being unable to formalize the notion of ``canonical isomorphism'' between
localizations $\locEl{R}{f}$ in a satisfactory way,
\Lean's \mathlib \cite{Lean/mathlib} consequently adopted
a non-standard take on localizations.
Ultimately, the definition of the structure sheaf got completely overhauled
using the Hartshorne approach.
Buzzard \etal{} argue in \cite[Sect. 3.4]{SchemesLean} that
even with the structure sheaf directly defined using univalence,
proving the sheaf property would run
into the same problems that they encountered.
As the equality/path obtained by an application of the univalence axiom would
still carry around the isomorphism in question,
it is a priori unclear what has actually been gained by working with paths,
as opposed to working with isomorphisms directly.
One of the main results of this paper is that on the contrary
we can use univalence in a genuinely helpful way to construct the structure sheaf
on basic opens and prove its sheaf property. This is achieved by shifting the
focus to $R$-algebras, where the canonical isomorphisms between localizations
become the center of contraction of the corresponding path spaces.
Indeed, the localizations $\locEl{R}{f}$ form a full subcategory of the category of
$R$-algebras that is posetal and equivalent to the poset of basic opens.

\subsection{Different univalent approaches to basic opens}
One of the main challenges of our formalization
was to solve the higher coherence issues when constructing
the structure presheaf on basic opens. These coherence issues arose
because the basic opens were defined as a subset of the Zariski lattice
(\ie\ as functions into propositions) using propositional truncation.
In constructive mathematics it is common to define subsets $X$ as
sets $A$ with an embedding $i:A\hookrightarrow X$ and
one can prove in HoTT/UF that these two notions of subsets are equivalent.
This raises the question whether one could define the type of basic opens
more directly, thus eliminating the coherence issues.

The basic opens can be defined as a quotient
on $R$, equating any $f$ and $g$ such that
$\sqrt{\langle f\rangle}=\sqrt{\langle g\rangle}$.
A first, now deprecated, formalization attempt
defined the structure sheaf on this type. However, in this
case we need to map from a set quotient into
an groupoid, which is notoriously hard. The general characterization of such maps
given by Kraus and von Raumer \cite[Thm. 13]{KrausQuotients} is not
easily applicable in this case.
As a result, we ended up working in $R$-algebras
because the contractibility of the path spaces between localizations solved
the coherence issues in this case as well. Rijke has since suggested, in private
communications, that the basic opens can be seen as the Rezk completion
\cite[Sec. 9.9]{HoTTBook} of $R$ as a poset category with the pre-order
$f\leq g$ given by inclusion $\sqrt{\langle f\rangle}\subseteq\sqrt{\langle g\rangle}$.
This could potentially be used for an alternative development
where coherence issues are avoided altogether.
%% The structure sheaf would then be induced by the universal property of the
%% Rezk completion .
%% The definition of the Rezk completion as representable presheaves,
%% which is the one implemented in the \CubicalAgda~library,
%% increases the universe level. However, the HIT definition proposed in the
%% HoTT-book \cite[Theorem 9.9.5]{HoTTBook}

\subsection{Towards constructive quasi-compact, quasi-separated schemes}
The structure sheaf, as constructed in this paper, lets us define
constructive affine schemes. This is of course only the first step towards a
formalization of \emph{constructive schemes}. Schemes are classically defined as a
special class of \emph{locally ringed spaces}. However, in the
constructive, predicative setting of~\cite{ConstrSchemes} we are confined to \emph{ringed lattices}, \ie\ distributive lattices equipped with a sheaf valued
in commutative rings.
These correspond to ringed coherent spaces. Maps between
those are maps of ringed spaces where the underlying continuous map is coherent.
Morphisms of schemes, however, are just morphisms of locally ringed spaces, \ie\
morpisms of ringed spaces that induce local morphisms on the stalks.
In general these two types of morphisms do not coincide.

Fortunately, the situation is well-behaved
for \emph{quasi-compact, quasi-separated} schemes,
a very important class of schemes that, in particular, encompasses all \emph{Noetherian}
schemes.\footnote{Deligne in fact argued that this class of schemes is actually
  sufficient for a lot of applications in algebraic geometry \cite{EGA4.5}.}
They are actually just the schemes where the underlying topological space is coherent.
Furthermore, if $X$ and $Y$ are quasi-compact, quasi-separated schemes,
for any morphism of locally ringed spaces
$(f,f^\sharp):(X,\mathcal{O}_X)\to(Y,\mathcal{O}_Y)$,
the underlying continuous map $f$ is coherent.
As pointed out in \cite{ConstrSchemes},
this was essentially already proved by Grothendieck \cite[Sec. 6.1]{EGA1}.
This makes the constructive lattice-based approach to quasi-compact, quasi-separated
schemes as worked out in \cite{ConstrSchemes} possible.

Such an approach still needs to be able to talk about morphisms of
quasi-compact, quasi-separated schemes, \ie\ morphisms of locally ringed spaces.
This problem is circumvented in \cite{ConstrSchemes} by considering
\emph{locally affine morphisms}. A locally affine morphism
is induced by ring homomorphisms on affine covers and it is a standard
exercise to show that for general schemes
this is equivalent to a morphism of locally ringed spaces.
For a formalization however, it could be advantageous to work with a
constructive reformulation of morphisms of locally ringed spaces.
Schuster discusses the right constructive, point-free notion of a
morphism of locally ringed spaces in the setting of formal topology in
\cite{SchusterZariski}.  Transferring this to a development based on
ringed lattices could lead to a constructive account of quasi-compact,
quasi-separated schemes closer to the usual classical
presentation and easier to formalize.

\subsubsection*{Acknowledgments}

We would like to thank Thierry Coquand for his continued feedback and
invaluable comments throughout this project.  We are also indebted to
Felix Cherubini for his comments and his work on the algebra section
of the \CubicalAgda~library, particularly for his ring solver.
Furthermore, we thank Peter LeFanu Lumsdaine, Egbert Rijke and the
participants of the ``Proof and Computation'' autumn school in
Fischbachau for our discussions.

\printbibliography

\end{document}